\documentclass[11pt,a4paper,twoside]{article}
\usepackage[english]{babel}
\usepackage[utf8x]{inputenc}
\usepackage{amsfonts}
\usepackage{amsmath}
\usepackage{amsthm}
\usepackage{amssymb}
\usepackage{latexsym}
\usepackage{footmisc}
\usepackage{hyperref}
\usepackage{graphics}
\usepackage{enumerate}
\usepackage{graphicx}
\usepackage{fancyhdr}
\usepackage{tikz}
\usetikzlibrary{snakes}
\usetikzlibrary{calc,positioning}
\usetikzlibrary{3d,calc}
\usepackage{wrapfig}
\usepackage{blindtext}
\newcommand{\mc}{\mathcal}
\newcommand{\mb}{\mathbb}

\theoremstyle{plain}
\newtheorem{theorem}{Theorem}[section]
\newtheorem{lemma}[theorem]{Lemma}

\newtheorem{corollary}[theorem]{Corollary}
\theoremstyle{definition}

\newtheoremstyle{named}{}{}{\itshape}{}{\bfseries}{.}{.5em}{\thmnote{#3's }#1}
\theoremstyle{named}

\begin{document}

\title{{\bf \large{TOPOLOGICAL GENERICITY OF NOWHERE DIFFERENTIABLE FUNCTIONS IN THE DISC ALGEBRA} \vspace{4mm}}}
\author{Alexandros Eskenazis}
\date{}
\maketitle

\begin{abstract}
\noindent In this paper we introduce a class of functions contained in the disc algebra $\mc{A}(D)$. We study functions $f \in \mc{A}(D)$, which have the property that the continuous periodic function $u = Ref|_{\mb{T}}$, where $\mb{T}$ is the unit circle, is nowhere differentiable. We prove that this class is non-empty and instead, generically, every function $f \in \mc{A}(D)$ has the above property. Afterwards, we strengthen this result by proving that, generically, for every function $f \in \mc{A}(D)$, both continuous periodic functions $u=Ref|_\mb{T}$ and $\tilde{u} = Imf|_\mb{T}$ are nowhere differentiable. We avoid any use of the Weierstrass function and we mainly use Baire's Category Theorem.
\end{abstract}

\section{Introduction}

The existence and properties of strange functions have attracted the interest of mathematicians since Weierstrass, who first gave an example of a function defined on $\mb{R}$ which was continuous, periodic but not differentiable at any real number. Several years later, Banach and Mazurkiewicz independently proved, in \cite{ban} and \cite{maz} respectively, that this is in fact the "generic" case:

\begin{theorem} [Banach-Mazurkiewicz]
Let J be a compact interval on $\mb{R}$ and $\mc{C}_{\mb{R}}(J)$ the space of real-valued continuous functions on J, endowed with the supremum norm. Then the class of all functions $f \in \mc{C}_{\mb{R}}(J)$ which are nowhere differentiable is residual in $\mc{C}_{\mb{R}}(J)$, i.e. it contains a $G_{\delta}$ and dense subset of $\mc{C}_{\mb{R}}(J)$.
\end{theorem}

\noindent The proof of the above theorem is a classical application of Baire's Category Theorem: in order to prove that there is (at least one) continuous but nowhere differentiable function, we instead consider the class $\mc{L}$ of all such functions, which is possibly empty. Then we give a set theoretic description of a subset $\mc{S}_0 \subseteq \mc{L}$ and using Baire's Theorem we prove that $\mc{S}_0$ is $G_{\delta}$ and dense in $\mc{C}_{\mb{R}}(J)$; therefore we know that $\mc{S}_0 \neq \emptyset$ without having found any particular element of $\mc{S}_0$ or $\mc{L}$. For the role of Baire's Theorem in Analysis we refer to \cite{grosse} and \cite{kahane}.

The context in which we will work is the disc algebra $\mc{A}(D)$, i.e. the set of all holomorphic functions on the open unit disk $D$ which extend continuously on $\overline{D}$. We endow $\mc{A}(D)$ with the supremum norm; thus $\mc{A}(D)$ becomes a Banach algebra and Baire's Theorem is at our disposal. We will present an analogue of Theorem 1.1 in this setting. For each function $f \in \mc{A}(D)$ we can naturally construct a $2\pi-$periodic function $h: \mb{R} \to \mb{C}$ defined by $h(\theta) = f(e^{i\theta})$; often, by abuse of notation, we write $f(\theta)$ instead of $h(\theta)$. From the definition of $\mc{A}(D)$ this function $h$ is of course continuous. Taking into consideration the above theorem, a natural question to ask is whether there are functions $f \in \mc{A}(D)$ such that the corresponding function $h$ is nowhere differentiable. Of course, if such a function $h$ is not differentiable at a point $\theta \in \mb{R}$, then either $Reh$ is not differentiable at $\theta$ or $Imh$ is not differentiable at $\theta$. Our main result is the following:

\begin{theorem}
Let $E$ be the class of all functions $f \in \mc{A}(D)$ such that the function $u=Ref|_{\mb{T}}$ is not differentiable at any point $\theta \in \mb{R}$. Then $E$ is residual in $\mc{A}(D)$, i.e. it contains a $G_{\delta}$ and dense subset.
\end{theorem}

\noindent As a corollary of this result, we prove that also the class $E_1 \subseteq E$ which consists of all the functions $f \in \mc{A}(D)$ such that neither $u=Ref|_{\mb{T}}$, nor $\tilde{u} = Imf|_{\mb{T}}$ are differentiable at any point $\theta \in \mb{T}$ is also residual. 

Our proof of the above theorem makes use only of Baire's Category Theorem and some elementary Fourier Analysis; it follows the pattern of the proof of Theorem 1.1. The essential point in our proof, is to show that we can approximate every periodic real valued function of class $C^1$ defined on $\mb{R}$ by a piecewise linear periodic continuous function in a manner such that a good approximation also holds for the harmonic conjugates. Afterwards, as done in the proof of Theorem 1.1, we approximate this piecewise linear function by another which has much bigger slopes. A simple calculation gives that this approximation also holds for the harmonic conjugates and thus the proof is complete. 

Simpler proofs of the same result could be given using the fact that $E$ is non empty, i.e. starting with a particular element $f_0 \in E$, which can easily be constructed using the classical Weierstrass function (see \cite{tit} for example). Also, an interesting question to ask is what results, analogue to the above, can one have in the case of several variables. In this case, the context is the polydisc algebra $\mc{A}(D^I)$ and, for a function $f \in \mc{A}(D^I)$, we want to examine the differentiability of its restriction $f|_{\mb{T}^I}$ as a function of the real vector $\theta = (\theta_i)_{i \in I} \in \mb{T}^I$. These questions will be answered in a future paper.

The above result establishes the topological genericity of nowhere differentiable functions in the disc algebra. After this result is proven, it is a reasonable question to ask whether this class is also generic in other senses. One could examine the dense lineability of this class, i.e. the existence of a dense linear subspace of $\mc{A}(D)$, every non-zero element of which is nowhere differentiable as well as the spaceability of this class, i.e. the existence of a closed infinite dimentional linear subpsace of $\mc{A}(D)$ with the above property. Finally, since $\mc{A}(D)$ is an algebra, it would be interesting to study the algebrability of this class, i.e. to examine whether the above questions hold in the case where linear subspaces were replaced by subalgebras. These will hopefully be examined in a future paper. Similar results in other spaces can be found in \cite{aron}, \cite{bayart} and \cite{gonz}.

\section{The proof of Theorem 1.2}

We will present the proof of the theorem by a series of lemmas following the outline of the proof of Theorem 1.1. For $n \in \mathbb{N}$ consider the sets
\begin{multline}
D_n = \Big\{ u \in \mc{C}_{\mb{R}}(\mb{T}) : \mbox{for every} \ \theta \in \mathbb{R} \ \mbox{there is a } y \in \left( \theta, \theta + \frac{1}{n} \right) \\ \mbox{such that } \ |u(y) - u(\theta)| > n | y -\theta| \Big\}
\end{multline}
and
\begin{equation}
E_n = \{f \in \mc{A}(D) :  Ref|_{\mb{T}} \in D_n \},
\end{equation}
where, as usual, we interpret functions defined on $\mb{T}$ as $2\pi-$periodic functions defined on $\mb{R}$.
\begin{lemma}
The following inclusion is valid:
\begin{equation}
\mc{S} = \bigcap_{n=1}^{\infty} E_n \subseteq E.
\end{equation}
\end{lemma}

\begin{proof}
Let $f \in \mc{S}$. Then, for every $n \in \mathbb{N}$ and every $\theta \in \mb{R}$, there is a $y_n = y_n(\theta)$ with $\theta < y_n < \theta + \frac{1}{n}$ such that
\begin{equation}
\left| \frac{u(y_n)-u(\theta)}{y_n-\theta} \right| > n,
\end{equation}
where $u=Ref|_{\mb{T}}$. Then, $\{y_n\}$ converges to $\theta$ and, from (4), we deduce that $u$ is not differentiable at $\theta$, i.e. $f \in E$.
\end{proof}

\noindent We will prove that $\mc{S}$ is $G_{\delta}$ and dense in $\mc{A}(D)$; from this our result will follow.

\begin{lemma}
For every $n \in \mb{N}$, the set $E_n$ is open in $\mc{A}(D)$, endowed with the supremum norm.
\end{lemma}

\begin{proof}
Let $n \in \mathbb{N}$. We will prove that $\mc{A}(D) \setminus E_n$ is closed in $\mc{A}(D)$. Let $\{f_m\}$ be a sequence in $\mc{A}(D) \setminus E_n$ and $f \in \mc{A}(D)$ such that $f_m \to f$ uniformly on $\overline{D}$. Since $f_m \notin E_n$, for each m, there is a $\theta_m \in \mb{R}$ such that 
\begin{equation}
\left| \frac{u_m(y) - u_m(\theta_m)}{y-\theta_m} \right| \leq n,
\end{equation}
for every $y \in \left( \theta_m, \theta_m + \frac{1}{n} \right)$, where $u_m = Ref_m|_{\mb{T}}$. Since each $u_m$ is $2\pi-$periodic, we can assume that $\theta_m \in [0,2\pi]$ for every $m$ and hence there is a subsequence $\{ \theta_{k_m} \}$ of $\{ \theta_m \}$ and a $\theta \in [0,2\pi]$ such that $\theta_{k_m} \to \theta$.

\smallskip

If $y \in \left( \theta, \theta + \frac{1}{n} \right)$, then for large enough $m$ it is true that $y \in \left( \theta_{k_m}, \theta_{k_m} + \frac{1}{n} \right)$. Thus, applying (5) for these indices $\{k_m\}$ and then letting $m \to \infty$ we have that
\begin{equation}
\left| \frac{u(y) - u(\theta)}{y-\theta} \right| \leq n,
\end{equation}
(again $u=Ref|_{\mb{T}}$) since the convergence of $\{f_m\}$ to $f$ is uniform. Hence (6) holds for every $y \in \left( \theta, \theta + \frac{1}{n} \right)$ and thus $f \notin E_n$. Thus, $\mc{A}(D) \setminus E_n$ is closed or equivalently $E_n$ is open.
\end{proof}

\begin{lemma}
For every $n \in \mb{N}$, the set $E_n$ is dense in $\mc{A}(D)$.
\end{lemma}

\noindent For the proof of Lemma 2.3 we will need another lemma:

\begin{lemma}
Let $u: \mb{R} \to \mb{R}$ a $2\pi-$periodic function of class $C^1$ and an $\varepsilon >0$. Then there exists a piecewise linear $2\pi-$periodic continuous function $u_0 : \mb{R} \to \mb{R}$ such that $\| u-u_0 \|_{\infty} < \varepsilon$ and if $u_0$ is linear at each interval defined by the partition $\{ 0 = t_0 < t_1 < ... < t_N = 2\pi \}$ then
\begin{equation}
\max_{1 \leq j \leq N} \sup \{ |u'(x) - u_0'(x)| : t_{j-1} < x < t_j \} < \varepsilon.
\end{equation}
\end{lemma}

\begin{proof}
Since both $u$ and $u'$ are continuous and periodic they are also uniformly continuous and thus there is a large $N>0$ such that if $x,y \in \mb{R}$ with $|x-y| \leq \frac{2\pi}{N}$, then
\begin{equation}
|u(x)-u(y)|< \frac{\varepsilon}{2} \ \ \ \mbox{and} \ \ \ |u'(x)-u'(y)|< \frac{\varepsilon}{2}.
\end{equation}
For $0 \leq j \leq N$ let $t_j$ be the points $\frac{2\pi j}{N}$ and $u_0 : [0,2\pi] \to \mb{R}$ the function which in each interval $[t_{j},t_{j+1}]$ connects linearly the points $(t_j, u(t_j))$ and $(t_{j+1},u(t_{j+1}))$. Since $u$ is $2\pi-$periodic, $u(0)=u(2\pi)$ and hence $u_0(0) = u_0(2\pi)$. Thus we can extend $u_0$ to a $2\pi-$periodic continuous function defined on $\mb{R}$. We will prove that $u_0$ fullfils the requirements of the lemma. From the periodicity of both $u$ and $u_0$ it is enough to examine the desired properties on $[0,2\pi]$.

\smallskip

Let $x \in [0,2\pi]$ and $1 \leq j \leq N$ such that $t_{j-1} \leq x \leq t_{j}$. Then $|x-t_{j-1}| \leq \frac{2\pi}{N}$ and thus $|u(x) - u(t_{j-1})| < \varepsilon/2$. But since $u(t_{j-1}) = u_0(t_{j-1})$ and $u_0$ is linear in the interval $[t_{j-1}, t_j]$ we deduce that
\begin{equation}
\begin{split}
|u(x) - u_0(x)| & \leq |u(x) - u(t_{j-1})| + |u_0(t_{j-1}) - u_0(x)| \\ & < \frac{\varepsilon}{2} + |u_0(t_{j-1}) - u_0(t_j) | \\ & < \frac{\varepsilon}{2} + \frac{\varepsilon}{2} = \varepsilon.
\end{split}
\end{equation}
Hence, indeed $\|u-u_0\|_{\infty} < \varepsilon$.

\smallskip

For the proof of (7), let $1 \leq j \leq N$ and $x \in (t_{j-1},t_j)$. Then,
\begin{equation}
u_0'(x) = \frac{u(t_j)-u(t_{j-1})}{t_j - t_{j-1}}
\end{equation}
and thus, there exists a $\xi \in (t_{j-1},t_j)$ such that $u_0'(x) = u'(\xi)$. Since $|x-\xi| < t_j - t_{j-1} = \frac{2\pi}{N}$, using (8) we deduce that $| u'(x) - u_0'(x) | < \varepsilon/2$ and (7) follows.
\end{proof}

\noindent {\it Proof of Lemma 2.3.} Let $f \in \mc{A}(D)$ and an $\varepsilon >0$. It is well known, that the set of polynomials is dense in $\mc{A}(D)$; thus, to prove the density of $E_n$, we may assume that $f$ is a complex polynomial. First, we approximate $f$ by a function $g \in \mc{A}(D)$ such that $Reg|_{\mb{T}}$ is piecewise linear. For this puprose we will use Lemma 2.4. Indeed, if $u=Ref|_\mb{T}$, then $u$ is of course of class $C^{\infty}$ and thus, from the above lemma there exists a piecewise linear function $u_0 \in \mc{C}_{\mb{R}}(\mb{T})$ such that 
\begin{equation}
\|u-u_0\|_{\infty} < \varepsilon \ \ \ \mbox{and} \ \ \ |u'(\theta)-u_0'(\theta)| < \varepsilon 
\end{equation}
for every $\theta$ in which $u_0$ is differentiable. Consider the function $U=u-u_0$. Using the Cuachy-Schwarz inequality, Parseval's identity (for the piecewise $C^1$ function $U$) and (11) we obtain that
\begin{equation}
\begin{split}
\sum_{k \in \mb{Z}} |\widehat{U}(k)| & \leq |\widehat{U}(0)| + \left( \sum_{k \in \mb{Z}\setminus \{0\}} \frac{1}{k^2} \right)^{1/2} \left( \sum_{k \in \mb{Z} \setminus \{0\}} \left| k \cdot \widehat{U}(k) \right|^2 \right)^{1/2} \\ &  < \varepsilon + \frac{\pi}{\sqrt{6}} \left( \sum_{k \in \mb{Z}\setminus\{0\}} \left| \widehat{U'}(k) \right|^2 \right)^{1/2} \\ & = \varepsilon + \frac{\pi}{\sqrt{6}} \left( \frac{1}{2\pi} \int_0^{2\pi} \left|U'(t)\right|^2 dt \right)^{1/2} \\ & < \varepsilon + \frac{\pi}{\sqrt{6}} \cdot \varepsilon = C_1 \cdot \varepsilon,
\end{split}
\end{equation}
for the constant $C_1 =1 + \frac{\pi}{\sqrt{6}} >0$. It is well known (see \cite{ahl} pp. 168-171) that $U$ can be extended continuously on $\overline{D}$ such that the restriction $U|_D$ is harmonic. Thus U has a harmonic conjugate $V$ on $D$, obtained by the mappings
\begin{equation}
r^k\cos(kx) \mapsto r^k\sin(kx) \ \ \ \mbox{and} \ \ \ r^k\sin(kx) \mapsto -r^k\cos(kx),
\end{equation}
where $0 \leq r < 1$, $x \in \mb{R}$ and $z=re^{ix} \in D$, for $k \in \mb{Z}$. The convergence of the series in (12) yields that $V$ can also extend continuously on $\overline{D}$: indeed, since from the above mapping we have $\widehat{V}(k) = -i \cdot \mbox{sign}(k) \cdot \widehat{U}(k)$, for $k \in \mb{Z} \setminus \{0\}$ and $\widehat{V}(0)=0$ we conclude that the series
\begin{equation}
\sum_{k \in \mb{Z}} \left| \widehat{V}(k) \right| < C_1 \cdot \varepsilon < \infty,
\end{equation}
and thus, the Fourier series of $V$ converges to $V$ uniformly on $\mb{T}$. Therefore, if $\tilde{u}= Imf|_\mb{T}$, the harmonic function $\tilde{u}-V$ extends continuously on $\overline{D}$ and is a harmonic conjugate of $u-U = u_0$. Hence, there exists a $g \in \mc{A}(D)$ such that $Reg|_\mb{T} = u_0$ and $Img|_\mb{T} = \tilde{u}_0 = \tilde{u}-V $. From the above inequalities and the Maximum Principle we deduce that:
\begin{equation}
\begin{split}
\|f-g\|_{\infty} & = \max\{ |f(\zeta)-g(\zeta)| : \zeta \in \mb{T} \} \\ &\leq \|u-u_0\|_{\infty} + \| \tilde{u}-\tilde{u}_0 \|_{\infty} = \|U\|_{\infty} + \|V\|_{\infty} \\ & \leq \sum_{k \in \mb{Z}} \left| \widehat{U}(k) \right| + \sum_{k \in \mb{Z}} \left| \widehat{V}(k) \right| < 2C_1 \cdot \varepsilon.
\end{split}
\end{equation}

\smallskip

We will now find a function $h \in E_n$ such that $\|g-h\|_{\infty} < K \cdot \varepsilon$  for some constant $K$ and the triangle inequality will imply the density of $E_n$ in $\mc{A}(D)$. We define
\begin{equation}
\ell_j = \frac{u_0(t_j)-u_0(t_{j-1})}{t_j-t_{j-1}}, \ \ \ \ j=1,2,...,N
\end{equation}
the slopes of the linear components of $u_0$ and consider a large $R \in \mb{N}$ such that the real number
\begin{equation}
m = \frac{2R\varepsilon}{\pi} > n + \max_{1 \leq j \leq N} |\ell_j|.
\end{equation}
Afterwards, we consider the continuous, $\frac{\pi}{R}-$periodic function $s: \mb{R} \to \mb{R}$ defined by
\begin{equation}
s(\theta) = \varepsilon \cdot \mbox{dist}\left( \theta , \frac{\pi}{R} \cdot \mb{Z} \right),
\end{equation}
where $\frac{\pi}{R} \cdot \mb{Z} = \{ ..., -\frac{\pi}{R}, 0, \frac{\pi}{R}, ... \}$. We can easily see that $s$ is piecewise linear with slopes $\pm m$. 

\smallskip

Of course, $s$ is also $2\pi-$periodic and thus we can consider the function $u_1 = u_0 + s \in \mc{C}_{\mb{R}}(\mb{T})$, which satisfies 
\begin{equation}
\|u_0-u_1\|_{\infty} = \|s\|_{\infty} = \frac{\pi}{2R} \cdot \varepsilon < 2\varepsilon.
\end{equation}
From $u_1$ we will get the function $h$ we are looking for. We will need two lemmas:

\begin{lemma}
The function $u_1$ belongs in the class $D_n$.
\end{lemma}

\begin{proof}
Let $\theta \in \mb{R}$ and some $\delta \in \left( 0, \frac{1}{n} \right)$ such that both $s$ and $u_0$ have constant derivatives in the interval $(\theta, \theta + \delta)$. Then, for $y$ in this interval it is true that $\theta < y < \theta + \frac{1}{n}$ and in addition:
\begin{equation}
|u_0(y)-u_0(\theta)| = | \ell_j | \cdot |y-\theta|,
\end{equation}
for some $1 \leq j \leq N$. Using (17), we deduce that
\begin{equation}
\begin{split}
|u_1(y)-u_1(\theta)| & \geq |s(y)-s(\theta)| - |u_0(y)-u_0(\theta)| \\ & = m \cdot |y-\theta| - |\ell_j | \cdot |y-\theta| \\ & > n \cdot |y-\theta|,
\end{split}
\end{equation}
i.e. $u_1 \in D_n$.
\end{proof}

\begin{lemma}
The above function $s$ satisfies the inequality
\begin{equation}
\sum_{k \in \mb{Z}} \left| \widehat{s}(k) \right| < L \cdot \varepsilon,
\end{equation}
for some absolute constant $L>0$.
\end{lemma}

\begin{proof}
We will compute the fourier coefficients $\widehat{s}(k)$: first of all, for $k \in \mb{Z}$ we observe that:
\begin{equation}
\begin{split}
\widehat{s}(k) & = \frac{1}{2\pi} \int_0^{2\pi} s(\theta) e^{-ik\theta} d\theta \\ & = \frac{1}{2\pi} \int_0^{2\pi} s\left( \theta + \frac{\pi}{R} \right) e^{-ik \left( \theta + \frac{\pi}{R} \right) } d\theta \\ & =e^{-\frac{ik\pi}{R}} \cdot \frac{1}{2\pi} \int_0^{2\pi} s(\theta) e^{-ik\theta} d\theta \\ & = e^{-\frac{ik\pi}{R}} \widehat{s}(k),
\end{split}
\end{equation}
where we have used the fact that $s$ is $\frac{\pi}{R}-$periodic. Thus, $\widehat{s}(k) \neq 0$ implies that $k=2\lambda R$, for some $\lambda \in \mb{Z}$. We now compute these coefficients: for $\lambda \in \mb{Z}\setminus \{0\}$
\begin{equation}
\begin{split}
\widehat{s}(2\lambda R) & = \frac{1}{2\pi} \int_{-\pi}^{\pi} s(\theta) e^{-2i\lambda R\theta} d\theta \\ & = \frac{1}{2\pi} \sum_{j=-R}^{R-1} \int_{j\pi/R}^{(j+1)\pi/R} s(\theta) e^{-2i\lambda R\theta} d\theta \\ & \stackrel{y=2R\theta}{=} \frac{1}{4\pi R} \sum_{j=-R}^{R-1} \int_{2j\pi}^{2(j+1)\pi} s\left( \frac{y}{2R} \right) e^{-i\lambda y} dy.
\end{split}
\end{equation}
Since the function inside the integral is $2\pi-$periodic we also have the equalities:
\begin{equation}
\begin{split}
\widehat{s}(2\lambda R) & = \frac{1}{4\pi R} \sum_{j=-R}^{R-1} \int_{-\pi}^{\pi} s \left( \frac{y}{2R} \right) e^{-i\lambda y} dy \\ & = \frac{1}{2\pi} \int_{-\pi}^\pi s \left( \frac{y}{2R} \right) e^{-i\lambda y} dy.
\end{split}
\end{equation}
But, for $-\pi \leq y \leq \pi$ it holds $-\frac{\pi}{2R} \leq \frac{y}{2R} \leq \frac{\pi}{2R}$ and thus, $s \left( \frac{y}{2R} \right) = \frac{m}{2R} \cdot |y|$. Thus, the computation gives:
\begin{equation}
\begin{split}
\widehat{s}(2\lambda R) & = \frac{1}{2\pi} \int_{-\pi}^\pi \frac{m}{2R} |y| e^{-i\lambda y} dy \\ & \stackrel{(16)}{=} \frac{\varepsilon}{2\pi^2} \int_{-\pi}^\pi |y| e^{-i\lambda y} dy \\ & = \frac{\varepsilon}{\pi} \cdot \frac{(-1)^\lambda-1}{\pi \lambda^2}.
\end{split}
\end{equation}
Thus,
\begin{equation}
\left| \widehat{s}(2\lambda R) \right| \leq \frac{C_2}{\lambda ^2} \cdot \varepsilon,
\end{equation}
for some absolute constant $C_2 >0$ and the result of the lemma follows, since $\widehat{s}(0) = \frac{\varepsilon \pi}{4R} < \varepsilon$ and hence:
\begin{equation}
\begin{split}
\sum_{k \in \mb{Z}} \left| \widehat{s}(k) \right| & = \sum_{\lambda \in \mb{Z}} \left| \widehat{s}(2\lambda R) \right| \\ & < \varepsilon + \sum_{t \in \mb{Z}\setminus\{0\}} \frac{C_2}{t^2} \cdot \varepsilon = L \cdot \varepsilon,
\end{split}
\end{equation}
for some constant $L>0$.
\end{proof}

\noindent {\it Proof of Lemma 2.3 (continued).} From (22), as we did after (12), we deduce that $s$ has a harmonic conjugate $\tilde{s}$ on $D$, which extends continuously on $\overline{D}$ and in addition satisfies the property $\| \tilde{s} \|_{\infty} < L \cdot \varepsilon$. We now define the function $h : \overline{D} \to \mb{C}$ by $h(z) = g(z) + (s(z) + i \tilde{s}(z))$, $|z| \leq 1$. From the above results, we have that $h \in \mc{A}(D)$, $Reh = u_0 + s = u_1 \in D_n$ and
\begin{equation}
\begin{split}
\|g-h\|_{\infty} & \leq \|u_1-u_0\|_{\infty} + \|\tilde{u}_1 - \tilde{u}_0 \|_{\infty} \\ & = \|s\|_{\infty} + \|\tilde{s}\|_{\infty} < K \cdot \varepsilon,
\end{split}
\end{equation}
for some constant $K>0$. Thus,
\begin{equation}
\|f-h\|_{\infty} \leq \|f-g\|_{\infty} + \|g-h\|_{\infty} < (2C_1+K) \cdot \varepsilon,
\end{equation}
and since $h \in E_n$ the proof of the lemma is complete.

\hfill$\Box$

\noindent {\it Proof of Theorem 1.2.} Since every $E_n$ is open and dense in $\mc{A}(D)$, from Baire's Category Theorem, we also deduce that their intersection $\mc{S}$ is $G_{\delta}$ and dense in $\mc{A}(D)$. Thus, from Lemma 2.1, Theorem 1.2 follows.

\hfill$\Box$

We close with a stronger version of Theorem 1.2:

\begin{corollary}
Let $E_1$ be the class of all functions $f \in \mc{A}(D)$ such that neither $u=Ref|_\mb{T}$, nor $\tilde{u} = Imf|_\mb{T}$ are differentiable at any point $\theta \in \mb{R}$. Then $E_1$ is residual in $\mc{A}(D)$.
\end{corollary}

\begin{proof}
Since the set $\mc{S}$ defined above is $G_{\delta}$ and dense in $\mc{A}(D)$ and multiplication by $i$ is a homeomorphism, the same holds for the set $i\mc{S}$. Thus, from Baire's Theorem the set $\mc{T} = \mc{S} \cap i \mc{S}$ is also $G_{\delta}$ and dense in $\mc{A}(D)$. We will prove that $\mc{T} \subseteq E_1$.

\smallskip

Let $f \in \mc{T}$. Since $f \in \mc{S}$, it is true that $u=Ref$, is nowhere differentiable from Lemma 2.1. But also, $f \in i\mc{S}$ and thus $-if \in \mc{S}$. Hence, we also conclude that $Re(-if) = Imf = \tilde{u}$ is nowhere differentiable.
\end{proof}

\vspace{3mm}

\noindent {\bf Aknowledgements:} The author would like to thank Professor Vassili Nestoridis for introducing him to the problem and Konstantinos Makridis for helpful communications.

\bigskip

\noindent {\scshape Alexandros Eskenazis:} Department of Mathematics, University of Athens, Panepistimioupolis, 157 84, Athens, Greece.\\
E-mail: \href{mailto:alex\_eske@hotmail.com}{\url{alex\_eske@hotmail.com}}

\end{document}